\documentclass[12pt]{amsart}
\usepackage{mathrsfs}
\usepackage{url}
\usepackage{color}
\usepackage[a4paper,asymmetric]{geometry}
\usepackage{latexsym}
\usepackage{amsthm}
\usepackage{amssymb}
\usepackage{amsfonts}
\usepackage{amsmath}

\newtheorem{theorem}{Theorem}[section]
\newtheorem{thm}[theorem]{Theorem}
\newtheorem{lemma}[theorem]{Lemma}
\newtheorem{lem}[theorem]{Lemma}
\newtheorem{proposition}[theorem]{Proposition}
\newtheorem{prop}[theorem]{Proposition}
\newtheorem{corollary}[theorem]{Corollary}

\theoremstyle{definition}

\newtheorem{defn}[theorem]{Definition}

\theoremstyle{remark}

\newtheorem{rem}[theorem]{Remark}
\numberwithin{equation}{section}

 \DeclareMathAlphabet{\mathpzc}{OT1}{pzc}{m}{it}

 \newcommand{\M}{\mathcal{M}}

\def\cal{\mathcal}
 \newcommand{\E}{\mathbb{E}}            % expectation
 \newcommand{\T}{\mathbb{T}}

 \newcommand{\Ll}{\langle}
 \newcommand{\Rr}{\rangle}

 \newcommand{\N}{\mathbb{N}}
 \newcommand{\R}{\mathbb{R}}
 \newcommand{\Z}{\mathbb{Z}}

 \newcommand{\FF}{\mathcal{F}}
 \newcommand{\PP}{\mathbb{P}}
 \newcommand{\mcl}{\mathcal}
 
 \newcommand{\Be}{\begin{equation}}
 \newcommand{\Ee}{\end{equation}}
 \newcommand{\Bs}{\begin{split}}
 \newcommand{\Es}{\end{split}}
  \newcommand{\Bes}{\begin{equation*}}
 \newcommand{\Ees}{\end{equation*}}
 \newcommand{\BT}{\begin{thm}}
 \newcommand{\ET}{\end{thm}}
 \newcommand{\Bp}{\begin{proof}}
 \newcommand{\Ep}{\end{proof}}
 \newcommand{\BL}{\begin{lem}}
 \newcommand{\EL}{\end{lem}}
 \newcommand{\BP}{\begin{proposition}}
 \newcommand{\EP}{\end{proposition}}
 \newcommand{\BC}{\begin{corollary}}
 \newcommand{\EC}{\end{corollary}}
 \newcommand{\BR}{\begin{rem}}
 \newcommand{\ER}{\end{rem}}
 \newcommand{\BD}{\begin{defn}}
 \newcommand{\ED}{\end{defn}}
 \newcommand{\BI}{\begin{itemize}}
 \newcommand{\EI}{\end{itemize}}

  \newcommand{\dif}{{\rm d}}

\def\div{   \textbf{div}}

\def\EE{\mathbb{E}}
\def\PP{\mathbb P}
\def\RR{\mathbb{R}} 

\def\cB{\mathcal{B}}

\def\cF{\mathcal{F}}
\def\cL{\mathcal{L}}\def\cM{{\cal M}}

\def\<{\left<}\def\>{\right>}
\def\({\left(}\def\){\right)}

\begin{document}
\title
[\ \ \ \ LDP for Non-linear monotone SPDEs ]{Large deviation principle of occupation measures for Non-linear monotone SPDEs}

\author[R. Wang]{Ran Wang}
\address{School of Mathematical Sciences, University of Science and Technology of China, Hefei, China.} \email{wangran@ustc.edu.cn}

 \author[J. Xiong]{Jie Xiong }
\address{Department of Mathematics, Faculty of Science and Technology,  University of  Macau, Taipa, Macau.}
\email{jiexiong@umac.mo}

\author[L. Xu]{Lihu Xu}
\address{Department of Mathematics, Faculty of Science and Technology,  University of  Macau, Taipa, Macau.}
\email{lihuxu@umac.mo}

%\thanks{The first author
%was supported by the M.I.U.R. research project Prin 2008 ``Deterministic and stochastic methods in the study of evolution problems''.
%The third author gratefully acknowledges the support by
%Junior program \emph{Stochastics} of Hausdorff Research Institute for
%Mathematics. His research is partially supported by the European Research Council under the European Union's
%Seventh Framework Programme (FP7/2007-2013) / ERC grant agreement nr. 258237.}
%\subjclass[2000]{}
%\keywords{}
%\date{}
%%% ----------------------------------------------------------------------
\maketitle
\begin{minipage}{140mm}
\begin{center}
{\bf Abstract}
\end{center}

Using  the hyper-exponential recurrence criterion,    a  large deviation principle for the occupation measure  is derived for  a class of non-linear monotone stochastic partial differential equations.    The main results are applied to many concrete SPDEs such
as   stochastic $p$-Laplace equation,  stochastic porous medium equation, stochastic fast-diffusion equation,  and even stochastic real Ginzburg-Landau equation driven by  $\alpha$-stable noises.

 %Our result indicates a phenomenon that strong dissipation beats heavy tailed noises to produce a large deviation principle, it seems to us that this phenomenon has not been reported in the known literatures.

\end{minipage}

\vspace{4mm}



%%% ----------------------------------------------------------------------

\section{Introduction}

The large time asymptotics for stochastic partial differential equations (SPDEs) have been studied in abundant literatures,
see \cite{DRRW, DPZ96,EH01, GM06,Hairer2008, Liu11, Ner08,PeZa07, WangFY} for ergodicity and \cite{Ge14, GLR11} for random attractor. An SPDE is ergodic means that the occupation measures of its solution converge to a unique invariant measure. It is natural to further ask whether the occupation measures satisfy a Donsker-Varadhan's large deviation principle \cite{DV, DZ, FW}, which gives an estimate on the probability that the occupation measures are deviated from the invariant measure.
\vskip0.3cm

Although Freidlin-Wentzell's small noise large deviation principles have been intensively studied for stochastic (partial) differential equations in recent years, see e.g. \cite{DZ,FW, KX96, XuZh09}, there seem only very few papers on  the  large deviations of Donsker-Varadhan type for  large time, see Gourcy's works \cite{Gou1, Gou2} for stochastic Burgers and Navier-Stokes equations and  Jak\u{s}i\`c et al. \cite{JNPS1, JNPS2}   for some dissipative SPDEs.
\vskip0.3cm

Wu \cite{Wu01} gave a criterion of large deviation principle of occupation measures for strong Feller and irreducible Markov processes, in which one needs to check hyper-exponential recurrence. Many techniques such as Bismut-Elworthy-Li formula \cite{ElLi94} and Wang type Harnack inequality  \cite{Wang2007,WangFYbook} have been developed for studying strong Feller property and irreducibility. However, hyper-exponential recurrence is a very strong condition and hard to be verified for stochastic PDEs.
\vskip0.3cm

This paper is devoted to studying LDP of the occupation measure for a family of monotone stochastic (partial) differential equations with some strong coercivity, which include some non-Lipschitz stochastic differential equations,  stochastic porous media equations, stochastic $p$-Laplace equation  see e.g. \cite{LiuJEE, RWZ, WangFY, ZhangX09} and so on. Our approach is via the hyper-exponential recurrence criterion by Wu. On the one hand, thanks to the monotonicity,
strong Feller property and irreducibility are established by coupling techniques \cite{WangFYbook} and Wang type Harnack inequality.  On the other hand, the strong  coercivity ($r>1$ in   (H3)  below) paves a way for us to proving hyper-recurrence condition.
\vskip0.3cm

For real Ginzburg-Landau equation driven by $\alpha$-stable noise, the Freidlin-Wentzell's small noise large deviation does not hold since $\alpha$-stable noise does not have second moment. However, as the time tends to infinity,
the strong coercive nonlinearity $N(x)=x-x^3$ wins the heavy tail effect to produce hyper-exponential recurrence and thus large deviation principle.
\vskip0.3cm

The paper is organized as follows. In Section 2, we   give  the models and main theorems. In Section 3,  the main theorems are applied to some concrete examples of stochastic (partial) differential equations.
 In Section 4,  we  recall the hyper-exponential criterion about the LDP for Markov processes with strong Feller property and  irreducibility. The proofs of the main theorems are given in  Section 5.

\section{The models and results}
Let $(H, \langle \cdot,\cdot\rangle_H, \|\cdot\|_H)$ be a separable Hilbert space and let $(V,\|\cdot\|_V)$ be a Banach space such that $V\subset H$ continuously and densely. Let $V^*$ be the dual space of $V$, it is well known
$$
V\subset H\subset V^*
$$
continuously and densely. If $_{V^*}\langle\cdot, \cdot\rangle_{V}$ denotes the dualization between $V^*$ and $V$, it follows that
$$
_{V^*}\langle z, v\rangle_{V}=
\langle z, v\rangle_{H}\ \ \ \ \text{for all } z\in H, v\in V.
$$
$(V, H, V^*)$ is called a {\it Gelfand triple}. In this paper, we always assume that $V$ is {\it compactly embedded} in $H$. Thus, there exists a constant $\eta>0$ such that
\Be\label{eq compact}
\|x\|_V\ge\eta\|x\|_H\ \ \ \text{for all } x\in V.
\Ee

Consider the following  stochastic differential equation  on $H$
\Be\label{eq SPDE}
\dif X(t)=A(X(t))\dif t+B(X(t))\dif W_t
\Ee
where $\{W_t\}_{t\ge0}$ is a cylindrical $Q$-Wiener process with $Q:=I$ on another separable Hilbert space $(U,\langle \cdot, \cdot\rangle_U)$ and being defined on a complete probability space $(\Omega, \cF,\PP)$ with normal filtration $\{\cF_t\}_{t\ge0}$,    $B$ is  in Hilbert-Schmidt space $\cL_2(U,H)$. {\it
We always assume that the measurable  functions
$A:V\rightarrow V^*$,  $B: V\rightarrow\cL_2(U, H)$
satisfy some conditions  such that  the equation \eqref{eq SPDE} has a unique continuous  solution in some sense.} For example, assume that
 \begin{itemize}
  \item[{\bf (H1)}] (Hemicontinuity) For all $v_1,v_2, v\in V,  \RR\ni s \mapsto {_{V^*}\langle }A(v_1+s v_2), v\rangle_{V}$ is continuous.
  \item[{\bf (H2)}] (Weak monotonicity) There exists $c_0\in \RR$ such that for all $v_1,v_2\in V$,
  $$
  2 _{V^*}\langle A(v_1)-A(v_2), v_1-v_2\rangle_{V}+\|B(v_1)-B(v_2)\|_{\cL_2(U,H)}^2\le c_0 \|v_1-v_2\|_H^2.
  $$
  \item[{\bf (H3)}] (Coercivity) There exist $r>0$ and $c_1,c_3\in \RR,c_2>0$ such that for all $v\in V$,
  $$
  2 _{V^*}\langle A(v), v\rangle_{V}+\|B(v)\|_{\cL_2(U,H)}^2\le c_1-c_2\|v\|_V^{r+1}+c_3\|v\|_H^2.
  $$
  \item[{\bf (H4)}](Boundedness) There exist  $c_4>0$ and $c_5>0$  such that for all $v\in V$,
  $$
  \|A( v)\|_{V*}\le c_4+c_5 \|v\|_V^{r},
  $$
  where $r$ is as in (H3).
  \end{itemize}

\BD
 A continuous $H$-valued  $\FF_t$-adapted process $\{X_t\}_{t\ge0}$ is called a solution of \eqref{eq SPDE}, if
 \Be\label{eq solution 1}
 \EE\left[\int_0^t\left(\|X(s)\|_V^{r+1}+\|X(s)\|_H^2\right)\dif s \right]<\infty, \ \ \  \forall t>0,
 \Ee
 and $\PP$-a.s.
 $$
 X(t)=X(0)+\int_0^tA(X(s))\dif s+\int_0^tB(X(s))\dif W(s),\ \  \forall t\ge0.
 $$
\ED

According to \cite{KR, PR}, under Conditions (H1)-(H4),  for any $X_0\in L^2(\Omega\rightarrow H;\cF_0;\PP)$, \eqref{eq SPDE} admits  a unique solution $\{X_t\}_{t\ge0}$.  Moreover, we have the following It\^o formula
\begin{align*}
\|X_t\|_H^2=&\|X_0\|_H^2+\int_0^t\left(2 _{V^*}\langle A(X_s), X_s\rangle_V+\|B(X_s)\|_{\cL_2(U,H)}^2 \right)\dif s\notag\\
&+2\int_0^t \langle X(s), B(X(s))\dif W(s)\rangle_H.
\end{align*}
Let $X_t^x$ be the solution of \eqref{eq SPDE} starting from $x$. Consider
the  associated transition semigroup
$$
P_t F(x):=\EE[F(X_t^x)],\ \ \ F\in \cB_b(H), \ x\in H, \ t>0,
$$
where $\cB_b(H)$ is the class of all bounded Borel measurable functions on $H$.
Throughout this paper, we always assume that the Markov semigroup  $P_t$ is strong Feller and irreducible in $H$, that is
\begin{itemize}
  \item[(a)] $P_t(\cB_b(H))\subset C_b(H)$ for all $t>0$ ({\it Strong Feller Property}),
  \item[(b)] $P_t1_U(x)>0$ for each $t>0, x\in H$ and any nonempty open set $U\subset H$ ({\it Irreducibility}).
\end{itemize}
There have been abundant literatures on  the study of strong Feller and irreducibility for SPDEs, see e.g. \cite{DPZ96}, \cite{PZ},  \cite{WangFYbook} and so on.

\vskip0.3cm
Let $\mathcal L_t$ be the occupation measure of the system \eqref{eq SPDE}  given by
\begin{equation}\label{e:occupation}
\mathcal L_t(A):=\frac1t\int_0^t\delta_{X_s}(A)\dif s \ \ \  \ \text{ for any measurable set } A,
\end{equation}
where $\delta_a$ is the Dirac measure at $a$. Then $\cL_t$ is in  $\mathcal M_1(H)$, the space of  probability measures on $H$. On  $\mathcal M_1(H)$, let  $\sigma(\mathcal M_1(H), \mathcal B_b (H))$ be the $\tau$-topology  of converence against measurable and bounded functions  which is much stronger than the usual weak convergence topology $\sigma(\mathcal M_1(H), C_b(H))$, where
  $C_b(H)$ is the space of all bounded continuous  functions on $H$.  See \cite{DV} or \cite[Section 6.2]{DZ}.
\vskip0.3cm

  Now, we are at the position to state our main results.

  %Based on the hyper-exponential recurrence criterion developed by Wu \cite{Wu01}, we  prove that the occupation measure $\mathcal L_t$ obeys an LDP under $\tau$-topology.

\begin{theorem}\label{thm main} Assume that (H3) holds with $r>1$ and   $P_t$ is strong Feller and irreducible in $H$.  Then the family $\PP^{\nu}(\mathcal L_T\in \cdot)$ as $T\rightarrow +\infty$ satisfies the large deviation principle  on $(\mathcal M_1(H),\tau)$, with speed $T$ and rate function $J$ defined by \eqref{rate func} below, uniformly for any initial measure $\nu$ in $\mathcal M_1(H)$.
 %More precisely, the following three properties hold:
%\begin{itemize}
 % \item[(a1)]  for any $a\ge0$, $\{\mu\in \mathcal M_1(H); J(\mu)\le a \}$ is compact in  $(\mathcal M_1(H),\tau)$;
%  \item[(a2)] (the lower bound) for any  open set $G$ in $(\mathcal M_1(H),\tau)$,
 %  $$
 %  \liminf_{T\rightarrow \infty}\frac1T\log\inf_{\nu\in\mathcal M_1(H)}\mathbb P^{\nu}(\mathcal L_T\in G)\ge -\inf_G J;
 %  $$
  %\item[(a3)](the upper bound) for any  closed set $F$ in  $(\mathcal M_1(H),\tau)$,
  % $$
  % \limsup_{T\rightarrow \infty}\frac1T\log\sup_{\nu\in\mathcal M_1(H)}\mathbb P^{\nu}(\mathcal % L_T\in F)\le -\inf_F J.
 %  $$
%\end{itemize}
\end{theorem}

\begin{theorem}\label{thm main 01}  Assume that (H3)   holds  with $r\in (0,1]$ and $c_3<0$,   $P_t$ is strong Feller and irreducible in $H$, and  $C_B:=\sup_{u\in H}\|B(u)\|_{L^2(U, H)}^2<\infty$. Let $\lambda_0\in (0,-\frac{c_3}{2C_B})$ and
$$
\cM_{\lambda_0, L}:=\left\{\nu\in\cM_1(H):\int_H e^{\lambda_0\|x\|_H^2}\nu(\dif x)\le L \right\}.
$$
Then the family $\PP^{\nu}(\mathcal L_T\in \cdot)$ as $T\rightarrow +\infty$ satisfies the large deviation principle on $(\mathcal M_1(H),\tau)$, with speed $T$ and rate function $J$ defined by \eqref{rate func} below, uniformly for any initial measure $\nu$ in $
\cM_{\lambda_0, L}$.

\end{theorem}

\begin{rem} For every $f:H\rightarrow \R$ measurable and bounded, as $\nu\rightarrow\int_{H}f\dif \nu$ is continuous with respect to (w.r.t.) the $\tau$-topology, then by our main results and the contraction principle (\cite[Theorem 4.2.1]{DZ}), $$\mathbb P^{\nu}\left(\frac{1}{T}\int_0^T f(X_s)\dif s\in \cdot\right)$$
satisfies the LDP on $\R$,   with the rate function given by
$$
J^f(r)=\inf\left\{J(\nu)<+\infty|\nu\in\mathcal M_1(H)\  \text{and } \int f\dif\nu=r \right\},\ \ \forall r\in\R.
$$
\end{rem}

\section{Examples}
In this section, we apply Theorem  \ref{thm main} and Theorem \ref{thm main 01} to some concrete examples of stochastic (partial) differential equations.

\subsection{Non-Lipschitz stochastic differential equations \cite{ZhangX09,RWZ}}

Consider the following stochastic differential equation
\Be\label{eq SDE}
\dif X_t=b(X_t) \dif t+\sigma(X_t)\dif W_t, \ \ \  X_0=x\in\RR^d,
\Ee
where $b:\RR^d\rightarrow\RR^d$ and $\sigma:\RR^d\rightarrow\RR^d\times\RR^d$ are continuous functions and $\{W_t\}_{t\ge0}$ is a $d$-dimensional standard Brownian motion defined on some complete probability space $(\Omega, \cF, (\cF_t)_{t\ge0}, \PP)$.

Let $\langle\cdot, \cdot\rangle$ denote the inner product in $\RR^d$, $|\cdot|$ the norm in $\RR^d$, and $\|\cdot\|_2$ the Hilbert-Schmit norm from $\RR^d$ to $\RR^d$. Assume that the continuous coefficients $b$ and $\sigma$ satisfy the following conditions:
\begin{itemize}
  \item[({\bf A1})](Monotonicity) There exists a $\lambda_0\in\RR$ such that for all $x,y\in\RR^d$,
  \begin{align*}
  2\langle x-y, b(x)-b(y)\rangle+\|\sigma(x)-\sigma(y)\|_2^2\le \lambda_0|x-y|^2(1\vee \log |x-y|^{-1}).
  \end{align*}
  \item[({\bf A2})](Growth of $\sigma$) There exists a $\lambda_1>0$ such that for all $x\in\RR^d$,
  $$
  \|\sigma(x)\|_2\le \lambda_1(1+|x|).
  $$
   \item[({\bf A3})](Non-degeneracy of $\sigma$) For some $\lambda_2>0$,
   $$
\sup_{x\in\RR^d}  \|\sigma^{-1}(x)\|_2\le \lambda_2.
   $$
  \item[({\bf A4})](One-side growth of $b$) There exist  $p>2$ and constants $\lambda_3>0, \lambda_4\ge0$ such that for all $x\in\RR^d$,
      $$
      2\langle x, b(x)\rangle+\|\sigma(x)\|_2^2\le -\lambda_3|x|^p+\lambda_4.
        $$
\end{itemize}

According to  Zhang \cite{ZhangX09}, under ({\bf A1})-({\bf A4}),  the equation \eqref{eq SDE} has a unique continuous strong solution $X_t$, whose semigroup $P_t$ is strong Feller and irreducible.

Take $V=H=\RR^d$. Then  (H3) automatically holds in Theorem \ref{thm main} with $r>1$.  Hence,    the assertion in Theorem \ref{thm main}  holds for  the solution to \eqref{eq SDE}.

\vskip0.3cm

\BR
Ren et al. \cite{RWZ} established the strong Feller property and irreducibility for the non-Lipschitz multivalued stochastic differential equation under the similar conditions. Thus,  the assertion in Theorem \ref{thm main} also holds for the  non-Lipschitz multivalued stochastic differential equation studied in \cite{RWZ}.

\ER

\subsection{Stochastic $p$-Laplace equation \cite{PR, LiuJMAA}}
Let $\Lambda$ be an open bounded domain in $\RR^d$ with smooth boundary. Consider the following Gelfand triple
$$
H_0^{1,p}(\Lambda)\cap L^q(\Lambda)\subset L^2(\Lambda)\subset
(H_0^{1,p}(\Lambda)\cap L^q(\Lambda))^*
$$
and the stochastic $p$-Laplace equation
\begin{equation}\label{e: p-Laplace}
\dif X_t=\left[\div(|\nabla X_t|^{p-2}\nabla X_t)-\gamma|X_t|^{q-2}X_t\right]\dif t+B\dif W_t, \ \ X_0=x,
\end{equation}
where $\max\{1, 2d/(d+2)\}<p\le 2< q$ and $\gamma>0$, $B$ is a Hilbert-Schmidt operator on $L^2(\Lambda)$ and $W_t$ is a cylindrical Wiener process on $L^2(\Lambda)$ w.r.t. a complete filtered probability space $(\Omega, \mathcal F, (\mathcal F_t)_{t\ge0}, \PP)$.

On the Sobolev space $H_0^{1,p}(\Lambda)$, consider the norm
$$
\|u\|_{1,p}:=\left(\int_{\Lambda}|\nabla u(\xi)|^p\dif \xi\right)^{1/p}.
$$
Since $\Lambda$ is bounded, by Poincar\'e inequality we know that $\|\cdot\|_{1,p}$ is equivalent to the classical Sobolev norm in $H_0^{1,p}(\Lambda)$. We denote the norm in $L^q(\Lambda)$ by $\|\cdot\|_q$ and the inner product in $L^2(\Lambda)$ by $\langle \cdot, \cdot\rangle_{L^2(\Lambda)}$.
 According to the Rellich-Kondrachov theorem, the embedding $H_0^{1,p}(\Lambda)\subset L^2(\Lambda)$ is compact.

Take $V:=
H_0^{1,p}(\Lambda)\cap L^q(\Lambda), H:= L^2(\Lambda)$.   According to \cite[Example 4.1.9]{PR}, Conditions (H1)-(H4) hold for \eqref{e: p-Laplace} with $r>1$.

Assume that $B$ is non-degenerate, that is, $Bx=0$ implies that $x=0$. We define the following intrinsic metric for $x\in H_0^{1,p}(\Lambda)$,
\begin{equation*} \|x\|_B:=
  \begin{cases}
    \|y\|_2 & \text{if } \ y\in L^2(\Lambda), By=x; \\
    +\infty & \text{otherwise.}
  \end{cases}
 \end{equation*}
If  there exist constants $\sigma\ge\frac 4p$ and $\delta>0$ such that
\begin{equation}\label{eq p-Lap Condition}
\|x\|^2_{1,p}\cdot\|x\|_2^{\sigma-2}\ge\delta\|x\|_B^{\sigma}, \ \  \ \ \forall x\in H_0^{1,p}(\Lambda),
 \end{equation}
 then  the associated Markov semigroup $P_t$ is strong Feller and irreducible on $L^2(\Lambda)$ by \cite[Theorem 1.2, Theorem 1.3]{LiuJMAA}. Hence,  the assertion in Theorem \ref{thm main} holds for  the solution to  \eqref{e: p-Laplace}.

\BR  According to \cite[Example 3.3]{LiuJEE}, when $p> 2$ and $q\in [1,p]$ in  \eqref{e: p-Laplace}, $P_t$ is  strong Feller,  and furthermore if $d=1$ and $B:=(-\Delta)^{-\theta}$ with $\theta\in (1/4,1/2]$, then $P_t$ is irreducible.    Hence, the assertion in Theorem \ref{thm main}   holds for  the solution to  \eqref{e: p-Laplace}.

 \ER

\vskip 0.3cm
\subsection{Stochastic generalized porous media equations \cite{DRRW, RRW}}\label{sec porous media}
Let $D\subset\RR^d$ be a bounded domain with smooth boundary, $\mu$ be the normalized Lebesgue measure on $D$. Let $\Delta$ be the Dirichlet Laplace on $D$ and $L:=-(-\Delta)^{\gamma}$ for $\gamma>0$.  For $r>1$, consider the Gelfand triple
$$
L^{r+1}(D, \mu)\subset H^{ \gamma}(D,\mu)\subset (L^{r+1}(D, \mu))^*,
$$
where $H^{\gamma}(\mu)$ is the completion of $L^2(\mu)$ under the norm
$$
|x|:=\left(\int_D|-(\Delta)^{-\gamma/2}x|^2\dif \mu \right)^{1/2}, \ \ \ x\in L^2(D, \mu),
$$
and $(L^{r+1}(D, \mu))^*$ is the dual space of $L^{r+1}(D, \mu)$ w.r.t. $H^{\gamma}(D, \mu)$. Let $V:=L^{r+1}(D, \mu)$ and $H:= H^{ \gamma}(D,\mu)$.  Notice that $L^{r+1}(D, \mu)$ is compactly embedded  in $H^{ \gamma}(D,\mu)$.

Let
$$
0<\lambda_1\le \lambda_2\le \cdots\lambda_n\rightarrow+\infty
$$
be the eigenvalues of $-\Delta$ including multiplicities with unite eigenfunctions $\{e_j\}_{j\ge1}$.
For  $r>1, c\ge0, q>1/2$,

\begin{equation}\label{e: function Porous}
\Psi(s):=s|s|^{r-1},\ \Phi(s):=c s,\ B(x)e_j:=b_i(x)j^{-q}e_j, \ j\ge1,
\end{equation}
where $\{b_j\}_{j\ge1}$ satisfies that
\begin{equation}\label{e: condition Porous}
    \begin{cases}  |b_i(u)-b_i(v)|\le b|u-v|, \ \ \  u, v\in H^{ \gamma}(D,\mu), \\
     \inf_{u\in H^{ \gamma}(D,\mu)}\inf_{i\ge1} b_i(u)>0. \\
                        \end{cases} \end{equation}

Consider the equation
\Be\label{e: Porous medium}
\dif X(t)=\left(L \Psi(X(t))+\Phi(X(t)) \right)\dif t+B(X(t))\dif W(t),
\Ee
where $W$ is a cylindrical Brownian motion on $H^{\gamma}(D,\mu)$ w.r.t. a complete filtered probability space $(\Omega, \cF, \{\cF_t\}_{t\ge0},\PP)$.
   According to \cite[Example 4.1.11]{PR},   Conditions (H1)-(H4) hold for system \eqref{e: Porous medium} and the constant $r>1$ in (H3).   If \eqref{e: condition Porous} holds and $\gamma\ge dp$, the Markov semigroup $P_t$  is  strong Feller by \cite[Example 6.1]{WangFY2015Nonlinear} or \cite[Example 3.3]{ZhangSQ}, and $P_t$ is irreducible by \cite[Example 3.3]{ZhangSQ}. Hence, the assertion in Theorem \ref{thm main}  holds for  the solution to  \eqref{e: Porous medium}.

\subsection{Stochastic fast-diffusion equations} As references for this equation, we refer e.g. to \cite{RRW, LiuWang}.

Let $D=(0,1)\subset \RR$ and let $\mu, H^{\gamma}, L, \Psi, \Phi, B$  be as in  section \ref{sec porous media} with $1/3<r<1$,  $\gamma=1$, $c<0$ and $q$ to be determined later. Let $V:=L^{r+1}(D,\mu)\bigcap H^1(D, \mu)$ with
$$
\|v\|_V=|v|_{L^{r+1}}+|v|, \  \ v\in V.
$$
We consider the equation \eqref{e: Porous medium} under the triple
$$
V\subset H^1(D,\mu)\subset V^*.
$$
According to \cite[Theorem 3.9]{RRW},  (H1)-(H4) hold with $1/3<r<1$ and  $c_3<0$.
Furthermore, for all $\theta \in \left(\frac{4}{r+1},\frac{6r+2}{r+1}\right)$  and $q\in \left(\frac12, \frac{3r+1}{\theta(r+1)}\right)$, the Markov semigroup $P_t$  is  strong Feller and irreducible  by \cite[Example 3.4]{ZhangSQ}.
Hence, the assertion in Theorem \ref{thm main 01}   holds for   stochastic fast-diffusion equations.

\subsection{Stochastic real Ginzburg-Landau equation driven by $\alpha$-stable noises  \cite{Xu13}}
Let $\T= \R/\Z$ be equipped with the usual Riemannian metric, and let $\dif \xi$
denote the Lebesgue measure on $\T$. For any $p\ge1$, let
$$
L^p(\T;\R):=\left\{x: \T\rightarrow\R; \|x\|_{L^p}:=\left(\int_\T |x(\xi)|^p \dif\xi\right)^{\frac1p}<\infty\right\}.
 $$
Denote
$$H:=\bigg\{x\in L^2(\T; \R); \int_\T x(\xi) \dif\xi =0\bigg\},$$
it is a separable real Hilbert space with inner product
$$\Ll x,y \Rr_H:=\int_\T x(\xi)y(\xi) \dif\xi,\ \ \ \ \ \forall \ x, y \in H.$$
For any $x\in H$, let
$$
\|x\|_H:=\|x\|_{L^2}=\left(\langle x,x\rangle_H\right)^{\frac12}.
$$
Let $\Z_*:=\Z \setminus \{0\}$. It is well known that $\left\{e_k; e_k=e^{i2 \pi k\xi}, \ k \in \Z_*\right\}$ is an orthonormal basis of $H$. For each $x \in H$,
it can be represented by  Fourier series
$$x=\sum_{k \in \Z_*} x_k e_k  \ \ \ \ {\rm with} \ \ \ x_k \in \mathbb C, \ x_{-k}=\overline{x_k}.$$

Let $\Delta$ be the Laplace operator on $H$. It is well known that
$D(\Delta)=H^{2,2}(\T) \cap H$. In our setting, $\Delta$ can be determined by the following
relations: for all $k \in \Z_*$,
$$\Delta e_k=-\gamma_k e_k\ \ \ \ {\rm with} \ \ \gamma_k=4 \pi^2 |k|^2,$$
with
$$H^{2,2}(\T) \cap H=\left\{x \in H; \ x=\sum_{k \in \Z_*} x_k e_k, \ \sum_{k \in \Z_*} |\gamma_k|^{2} |x_k|^2<\infty\right\}.$$
Denote
$$A=-\Delta, \ \ \ \ D(A)=H^{2,2}(\T) \cap H.$$
Define the operator $A^{\sigma}$ with $\sigma \ge 0$ by
$$A^\sigma x=\sum_{k \in \Z_*} \gamma_k^{\sigma} x_ke_k, \ \ \ \ \ \ x \in D(A^\sigma),$$
where $\{x_k\}_{k \in \Z_*}$ are the Fourier coefficients of $x$, and
$$D(A^\sigma):=\left\{x \in H: \ x=\sum_{k \in \Z_*} x_k e_k, \sum_{k \in \Z_*} |\gamma_k|^{2 \sigma} |x_k|^2<\infty\right\}.$$

Given $x \in D(A^\sigma)$, its norm is
$$\|A^\sigma x\|_H:=\left(\sum_{k \in \Z_*} |\gamma_k|^{2 \sigma} |x_k|^2\right)^{1/2}.$$
For $\sigma>0$, let
$$
H_{\sigma}:=D(A^\sigma), \ \  \ \  \|x\|_{H_{\sigma}}:=\|A^\sigma x\|_H.$$
Then, $H_{\sigma}$ is densely and compactly embedded in $H$.  Particularly, let
$$V:=D(A^{1/2}).$$

We shall study  $1$D stochastic Ginzburg-Landau equation on $\T$ as the following
\begin{equation} \label{e:XEqn}
\begin{cases}
\dif X_t + A X_t\dif t= N(X_t) \dif t + \dif L_t, \\
X_0=x,
\end{cases}
\end{equation}
where
\begin{itemize}
\item[(i)] the nonlinear term $N$ is defined by
\begin{equation*} \label{e:NonlinearB}
N(u)= u-u^3 , \ \ \ \ \ u \in H.
\end{equation*}
\item[(ii)] $L_t=\sum_{k \in \Z_*} \beta_k l_k(t) e_k$ is an $\alpha$-stable process on $H$ with
$\{l_k(t)\}_{k\in \Z_*}$ being i.i.d. 1-dimensional symmetric $\alpha$-stable process sequence with $\alpha>1$. Moreover, we assume that there exist some $C_1, C_2>0$ so that $C_1 \gamma_k^{-\beta} \le |\beta_k| \le C_2 \gamma_k^{-\beta}$ with $\beta>\frac 12+\frac 1{2\alpha}$.
\end{itemize}

\BD We say that a predictable $H$-valued stochastic process $X=(X_t^x)$ is a mild solution to Eq. \eqref{e:XEqn} if, for any $t\ge0, x\in H$, it holds ($\mathbb P$-a.s.):
\begin{equation}\label{e: mild solution}
X^x_t(\omega)=e^{-At} x+\int_0^t e^{-A(t-s)} N(X_s^x(\omega))\dif s+\int_0^t e^{-A(t-s)} \dif L_s(\omega).
\end{equation}
\ED

The following  properties for the solutions   can be found in \cite{Xu13, WXX}.
\begin{prop}[\cite{Xu13, WXX}] \label{Thm Xu 13} Assume that $\alpha \in (3/2,2)$ and $\frac 12+\frac{1}{2\alpha}<\beta<\frac 32-\frac{1}{\alpha}$, the following statements hold:
\begin{enumerate}
\item    For every $x \in H$ and $\omega \in \Omega$ a.s.,
Eq. \eqref{e:XEqn} admits a unique mild solution $X=(X^x_t)_{t\ge0,x\in H} \in D([0,\infty);H) \cap D((0,\infty);V)$.

\item  $X$ is a Markov process, which is strong Feller and irreducible in $H$, and
 $X$ admits a unique invariant measure $\pi$.
  %Furthermore,  there exist some positive constants $M>1, \rho\in(0,1), \theta>0$ satisfying that  $\int \Psi\dif \pi<+\infty$, where $\Psi(x):= (M+\|x\|_H^2)^{1/2}$,
  %and $\pi$ is exponentially ergodic in the sense that
%\begin{equation*}
%\sup_{|f|\le \Psi}\left|P_tf(x)-\int f\dif \pi \right|\le \theta \Psi(x)\cdot \rho^t\ \ \  \forall x\in H,  t\ge0.
%\end{equation*}
\end{enumerate}
\end{prop}

Moreover, we prove that the system converges to its invariant measure $\mu$ with exponential rate under a topology stronger than the total variation, and the occupation measure $\mathcal L_t$ obeys the moderate deviation principle by constructing some Lyapunov test functions in our previous paper \cite{WXX}.
\vskip0.3cm

We shall establish in this paper the large deviation principle for the occupation measure $\mathcal L_t$ in the next theorem.
\begin{theorem}\label{thm main 2} Assume that $\alpha \in (3/2,2)$ and  $\frac 12+\frac{1}{2\alpha}<\beta<\frac 32-\frac{1}{\alpha}$.
Then the family $\PP^{\nu}(\mathcal L_T\in \cdot)$ as $T\rightarrow +\infty$ satisfies the large deviation principle with respect to the $\tau$-topology, with  a good rate function, uniformly for any initial measure $\nu$ in $\mathcal M_1(H)$.
\end{theorem}

%\section{The proof of main theorems}

\section{General results about large deviations}
In this section, we recall some general results on the Large Deviation Principle for strong Feller and irreducible   Markov processes  from \cite{Wu00, Wu01}.

Let $E$ be a Polish metric space. Consider a general $E$-valued c\`adl\`ag  Markov process
$$
\left(\Omega, \{\mathcal F_t\}_{t\ge0}, \mathcal F, \{X_t(\omega)\}_{t\ge0}, \{\mathbb P_x\}_{x\in E}\right),
$$
where
\begin{itemize}
  \item  $\Omega=D( [0,+\infty); E)$, which is the space of the c\`adl\`ag functions from $[0,+\infty)$ to $E$ equipped with the  Skorokhod topology; for each $\omega\in \Omega$, $X_t(\omega)=\omega(t)$;
  \item $\FF_t^0=\sigma\{X_s: 0\le s\le t\}$ for any $t\ge 0$ (nature filtration);
  \item $\FF=\sigma\{X_t: t\ge0\}$ and $\PP^{x}(X_0=x)=1$.
\end{itemize}
 Hence, $\PP^{x}$ is the law of the Markov process with initial state $x\in E$.  For any initial measure $\nu$ on $E$, let $\PP^{\nu}(\dif \omega):=\int_E \PP^{x}(\dif \omega)\nu(\dif x)$. Its transition probability is denoted by $\{P_t(x, dy)\}_{t\ge0}$.

For all $f\in\mathcal   B_b(E)$, define
$$
P_tf(x):=\int_E P_t(x, \dif y)f(y)  \ \ \ \text{for all } t\ge0, x\in E.
$$
For any $t>0$, $P_t$ is said to be {\it strong Feller} if $P_t\varphi\in C_b(E)$ for any $\varphi\in \mathcal B_b(E)$; $P_t$ is {\it irreducible} in $E$ if $P_t1_O(x)>0$ for any $x\in E$ and any non-empty open subset $O$ of $E$. $\{P_t\}_{t\ge0}$ is  {\it accessible }to $x \in E$,  if the resolvent $\{\mcl R_{\lambda}\}_{\lambda>0}$ satisfies
$$\mcl R_{\lambda}(y, \mcl U):= \int_0^\infty e^{-\lambda t}P_t(y, \mcl U) \dif t>0 , \ \ \forall \lambda>0$$
for all $y \in E$ and all neighborhoods $\mcl U$ of $x$. Notice that the accessibility of $\{P_t\}_{t\ge0}$ to any $x\in E$ is   the so called {\it topological transitivity} in Wu \cite{Wu01}.

\vskip0,3cm

The empirical measure of level-$3$ (or process level) is given by
$$
R_t:=\frac1t\int_0^t \delta_{\theta_s X}\dif s
$$
where $(\theta_sX)_t=X_{s+t}$ for all $t, s\ge0$ are the shifts on $\Omega$. Thus, $R_t$ is a random variable valued in $\mathcal M_1(\Omega)$, the space of all probability measures on $\Omega$.

The level-$3$ entropy functional of Donsker-Varadhan $H:\mathcal M_1(\Omega)\rightarrow [0,+\infty]$ is defined by
\begin{equation*} \label{e: DV}
H(Q):=\begin{cases}
 \mathbb E^{\bar Q}h_{\mathcal F_1^0}(\bar Q_{w(-\infty,0]};\mathbb P_{w(0)}) & \text{if } Q\in \mathcal M_1^s(\Omega);  \\
+\infty & \text{otherwise},
\end{cases}
\end{equation*}
where
 \begin{itemize}
   \item $\mathcal M_1^s(\Omega)$ is the subspace of $\mathcal M_1(\Omega)$, whose  elements are moreover stationary;
   \item $\bar Q$ is the unique stationary extension of $Q\in \mathcal M_1^s(\Omega)$ to $\bar \Omega:=D(\mathbb R; E)$; $\mathcal F_t^s=\sigma\{X(u); s\le u\le t\},\forall  s,t\in\R, s\le t$;
   \item $\bar Q_{w(-\infty,t]}$ is the regular conditional distribution of $\bar Q$ knowing $\mathcal F_t^{-\infty}$;
   \item $h_{\mathcal G}(\nu;\mu)$ is the usual relative entropy or Kullback information of $\nu$ with respect to $\mu$ restricted to the $\sigma$-field $\mathcal G$, given by
\begin{equation*}
h_{\mathcal G}(\nu;\mu):=\begin{cases}
  \int\frac{\dif\nu}{\dif\mu}|_{\mathcal G} \log\left(\frac{\dif\nu}{\dif\mu}|_{\mathcal G}\right) \dif\mu & \text{ if }  \nu\ll \mu \text{ on } \ \mathcal G;  \\
+\infty & \text{otherwise}.
\end{cases}
\end{equation*}
 \end{itemize}

The level-$2$ entropy functional $J: \mathcal M_1(E)\rightarrow [0, \infty]$ which governs the LDP in our main result is
\begin{equation}\label{rate func}
J(\mu)=\inf\{H(Q)| Q\in \mathcal M_1^s(\Omega) \  \ \text{and } Q_0=\mu\}, \ \ \ \ \forall \mu\in \mathcal M_1(E),
\end{equation}
where $Q_0(\cdot)=Q(X_0\in \cdot)$ is the marginal law at $t=0$.

   Recall the following hyper-exponential recurrence criterion for LDP established by Wu \cite[Theorem 2.1]{Wu01}, also see Gourcy \cite[Theorem 3.2]{Gou2}.

\vskip0.3cm
For  any measurable set $K\in E$, let
\begin{equation}\label{stopping time}
\tau_K=\inf\{t\ge0 \ \text{ s.t.}\   X_t\in K\},\ \ \ \tau_K^{(1)}=\inf\{t\ge1\  \text{ s.t.}\ X_t\in K\}.
\end{equation}

\begin{theorem}\cite{Wu01}\label{thm Wu}
Let $\mathcal A\subset \mathcal M_1(E)$ and assume that
\begin{equation*}\label{condition 1}
\{P_t\}_{t\ge0} \text{ is strong Feller and topologically irreducible on  } E.
\end{equation*}
If for any $\lambda>0$, there exists some compact set $K\subset \subset E$, such that
\begin{equation}\label{condition 2}
\sup_{\nu\in\mathcal A}\E^{\nu}e^{\lambda\tau_K}<\infty, \ \  \text{and} \ \ \
\sup_{x\in K}\E^{x}e^{\lambda\tau_K^{(1)}}<\infty.
\end{equation}
 Then the family $\mathbb P^{\nu}(\mathcal L_t\in\cdot)$ satisfies the LDP on $\mathcal M_1(E)$ w.r.t. the $\tau$-topology with the rate function $J$ defined by \eqref{rate func}, and uniformly for initial measures $\nu$ in the subset $\mathcal A$. More precisely, the following three properties hold:
\begin{itemize}
  \item[(a1)] for any $a\ge0$, $\{\mu\in \mathcal M_1(E); J(\mu)\le a \}$ is compact in  $(\mathcal M_1(E),\tau)$;
  \item[(a2)] (the lower bound) for any open set $G$ in $(\mathcal M_1(E), \tau)$,
   $$
   \liminf_{T\rightarrow \infty}\frac1T\log\inf_{\nu\in\mathcal A}\mathbb P^{\nu}(\mathcal L_T\in G)\ge -\inf_G J;
   $$
  \item[(a3)](the upper bound) for any  closed set $F$ in $(\mathcal M_1(E), \tau)$,
   $$
   \limsup_{T\rightarrow \infty}\frac1T\log\sup_{\nu\in\mathcal A}\mathbb P^{\nu}(\mathcal L_T\in F)\le -\inf_F J.
   $$
\end{itemize}

\end{theorem}

\section{The proof of main theorems}
In this section, we   prove  the main  theorems of this paper  according to Theorem \ref{thm Wu}.

\begin{proof}[Proof of Theorems \ref{thm main} and  \ref{thm main 01}]
\ Let $\{X_t\}_{t\ge0}$ be the solution to Eq. \eqref{eq SPDE}. Since $X$ is strong Feller and irreducible in $H$, according to Theorem \ref{thm Wu}, to prove Theorems \ref{thm main} and  \ref{thm main 01}, we need to prove that the hyper-exponential recurrence condition \ref{condition 2} is fulfilled. The verification of this condition  will be given  in the following two subsections.
\end{proof}

\subsection{The proof of Theorem \ref{thm main}}

 \begin{lemma}\label{lem expect bound}
There exist some constant $C$ and $t_1\in[1,2]$ such that for any $X_0\in L^2(\Omega\rightarrow H;\cF_0,\PP)$,
$$
\EE\left[\|X_{t_1}\|_{V}  \right]\le C.
$$
\end{lemma}

\begin{proof}

  {\bf Step 1.} We first prove that {\it
there exists a constant $C$ such that for any initial value $X_0\in L^2(\Omega\rightarrow H;\cF_0,\PP)$,
\begin{equation}\label{eq Lem 1 1}
\sup_{t\ge1}\EE\left[ \|X_t\|_H^2\right]\le  C.
\end{equation}
 }

By It\^o's formula, we have for any $t> s\ge 0$,
\begin{align}\label{eq Lemma 1 2}
\|X_t\|_H^2 =&\|X_s\|_H^2 +\int_s^t\left(2 _{V^*}\langle A(X_u), X_u\rangle_{V}+\|B(X_u)\|_{\cL_2(U,H)}^2 \right)\dif s\notag\\
   &+2\int_s^t\langle X_u, B(X_u)\dif W_u\rangle_H.
\end{align}
It is easy to check from \eqref{eq solution 1} that $\left\{\int_s^t\langle X_u, B(X_u)\dif W_u\rangle_H:t\ge s \right\}$ is a martingale. Taking the expectation of the both sides of \eqref{eq Lemma 1 2}, we have
 \begin{align}\label{eq Lemma 1 3}
 \EE\left[\|X_t\|_H^2\right] =&\EE\left[\|X_s\|_H^2\right] +\EE\left[\int_s^t\left(2 _{V^*}\langle A(X_u), X_u\rangle_{V}+\|B(X_u)\|_{\cL_2(U,H)}^2 \right)\dif s\right]
\end{align}
By   Condition (H3) and \eqref{eq compact}, we have
\begin{align}\label{eq Lemma 1 4}
 \EE\left[\|X_t\|_H^2\right] \le& \EE\left[\|X_s\|_H^2\right] + c_1 (t-s)-c_2 \int_s^t \EE\left[\|X_u\|_V^{r+1}\right]\dif u+c_3\int_s^t \EE\left[\|X_u\|_H^{2}\right]\dif u\notag\\
 \le& \EE\left[\|X_s\|_H^2\right] + C_1 (t-s)- C_2\int_s^t\EE\left[\|X_u\|_H^{r+1}\right]\dif u.
\end{align}

 Taking $s=0$ in the above inequality and according to  Lemma  \ref{lem p} below,  there exists a constant $C$ independent of $X_0$ satisfying  \eqref{eq Lem 1 1}.

\vskip0.3cm
{\bf Step 2.} By \eqref{eq Lemma 1 3} and Condition (H3), we obtain that
\begin{align}\label{eq Lemma 1 5}
 \EE\left[\|X_t\|_H^2\right]+c_2 \int_s^t \EE\left[\|X_u\|_V^{r+1}\right]\dif u\le \EE\left[\|X_s\|_H^2\right] +  \EE\left[\int_s^t\left(c_1+c_3\|X_u\|_H^2 \right)\dif u \right],
\end{align}
which, together with   \eqref{eq Lem 1 1}, further gives
\begin{equation*}
 \int_1^2\EE\left[\|X_u\|_{V}^{r+1}\right]\dif u \le C,
\end{equation*}
where the constant $C$ is independent of $X_0$. Therefore, we can  conclude that  there exists  $t_1\in[1,2]$ such that $\EE\left[\|X_{t_1}\|_V\right]\le  C$.

The proof is complete.
\end{proof}

 \subsubsection{The hyper-exponential recurrence}
Next, we will verify the hyper-exponential recurrence condition \eqref{condition 2}.
\vskip0.3cm
By the Markov property of $X$ and Lemma \ref{lem expect bound}, we know that there exists a sequence of times $\{t_n;n\ge1\}$ such that $t_n\in [2n-1,2n]$ and
\Be\label{eq X tn}
\EE\left[\|X_{t_n}\|_V\right]\le  C,
\Ee
where $C$ is the constant in Lemma \ref{lem expect bound}.

\vskip0.3cm

For any $M>0$, define the hitting time  of $\{X_{t_n}\}_{n\ge1}$:
\begin{equation}
\tau_M=\inf\{k\ge1:\|X_{k}\|_{V}\le M\}.
\end{equation}
 Let
 $$
 K:=\{x\in  V: \|x\|_{V}\le M\}.
 $$
 Clearly, $K$ is compact in $H$. Recall the definitions of  $\tau_K$ and $\tau_K^{(1)}$ in \eqref{stopping time}. It is obvious  that
\begin{equation}
\tau_K\le \tau_M,\ \ \ \ \ \tau_K^{(1)}\le \tau_M.
\end{equation}
This fact, together with the following important theorem,  implies the hyper-exponential recurrence condition  \eqref{condition 2}.

\begin{theorem}\label{thm hyper-exp} For any $\lambda>0$,  there exists a constant $M$  such that
$$
\sup_{\nu\in\mathcal M_1(H)}\E^{\nu}[e^{\lambda\tau_M}]<\infty.
$$
\end{theorem}
\begin{proof}
For any $n\in\N$, let
$$
B_n:=\left\{\|X_{t_{j}}\|_{V}>M: j=1,\cdots,n\right\}=\{\tau_M>n\}.
$$
By the Markov property of $\{X_{t_n} \}_{n\in\N}$, Chebychev's inequality and \eqref{eq X tn}, we obtain that for any $\nu\in \M_1(H)$,
\begin{align*}
\PP^{\nu}(B_n)=&\PP^{\nu}(B_{n-1})\cdot\PP^{\nu}(B_n|B_{n-1})\\
\le&\PP^{\nu}(B_{n-1})\cdot \E^{\nu} \left\{\frac{\E^{X_{t_{n-1}}}\left[\|X_{t_n}\|_{V}\right]}{M}\right\}\\
\le & \PP^{\nu}(B_{n-1})\cdot\frac{C}{M},
\end{align*}
where $C$ is the constant in Lemma \ref{eq Lemma 1 5}.

By the induction, we have for any $n\ge0$,
$$
\PP^{\nu}(\tau_M>n)=\PP^{\nu}(B_n)\le \left(\frac{C}{M}\right)^n.
$$
This inequality, together with Fubini's theorem, implies that for any $\lambda>0,\nu\in\M_1(H)$,
\begin{align*}
\E^{\nu}\left[ e^{\lambda\tau_M}\right]=&\int_0^{\infty}\lambda e^{\lambda t}\PP^{\nu}(\tau_M>t)\dif t\\
\le &\sum_{n=0}^{\infty}\lambda e^{\lambda (n+1)}\PP^{\nu}(\tau_M>n)\\
\le &\sum_{n=0}^{\infty}\lambda e^{\lambda (n+1)}\left(\frac{C}{M}\right)^n,
\end{align*}
which is finite as $M>C e^{\lambda}$.

The proof is complete.
\end{proof}

\subsection{The proof of Theorem \ref{thm main 01}}  The  strategy to verify the  hyper-exponential recurrence condition in this section  is inspirited by Gourcy \cite{Gou1, Gou2}.
 \vskip0.3cm

First, we establish the following crucial exponential estimate for the solution.

\BP\label{prop exp est}Under the conditions of Theorem \ref{thm main 01},  for any fixed $0<\lambda_0<-c_3/(2C_B)$ and   $x\in H$, the process $X$ satisfies that for any $t>0$,
\Be\label{eq exp1}
\EE^x\left[\exp\left(\frac{\lambda_0c_2}{2}\int_0^t \|X_s\|_V^r\dif s  \right)  \right]\le e^{\lambda_0 c_1t}\cdot e^{\lambda_0\|x\|_H^2}.
\Ee
\EP
\begin{proof}
Let $$Y_t:=\|X_t\|_H^2+\frac{c_2}{2}\int_0^t\|X_s\|_V^r\dif s.$$
By It\^o's formula and (H3), we have
\begin{align*}
\dif Y_t   =& \left(2 _{V^*}\langle A(X_t), X_t\rangle_{V}+\|B(X_t)\|_{L_2(U, H)}^2 \right)\dif t+2\langle X_t, B(X_t)\dif W_t\rangle_H+\frac{c_2}{2}\|X_t\|_V^r\dif t\notag\\
   \le& \left(c_1-\frac{c_2}{2}\|X_t\|_V^r+c_3\|X_t\|_H^2 \right)\dif t+2\langle X_t, B(X_t)\dif W_t\rangle_H.
 \end{align*}
In the same spirit, denoting by $\dif [Y,Y]_t$ the quadratic variation process of a semimartingale $Y$, we can also compute with the  It\^o's formula,
\begin{align*}
 \dif e^{\lambda_0Y_t}   =&  e^{\lambda_0Y_t} \left\{\lambda_0 \dif Y_t+\frac{\lambda_0^2}{2}\dif [Y,Y]_t\right\}\notag\\
 \le & \lambda_0 e^{\lambda_0Y_t} \left(c_1-\frac{c_2}{2}\|X_t\|_V^r+c_3\|X_t\|_H^2+2\lambda_0\|B (X_t)\|_{\cL_2(U,H)}^2\cdot\|X_t\|_H^2 \right)\dif t\notag\\
 &+ 2\lambda_0e^{\lambda_0Y_t}\langle X_t, B(X_t)\dif W_t\rangle_H. \notag
 \end{align*}
Let $Z_t:=e^{-\lambda_0 c_1 t}e^{\lambda_0 Y_t}$.
By It\^o's formula again, we obtain that when $\lambda_0<-c_3/(2C_B)$,
\begin{align*}
\dif Z_t=&e^{-\lambda_0 c_1 t}\dif e^{\lambda_0 Y_t}+e^{\lambda_0 Y_t}\dif e^{-\lambda_0 c_1 t}\\
\le & \lambda_0 Z_t\left( -\frac{c_2}{2}\|X_t\|_V^r+c_3\|X_t\|_H^2+2\lambda_0 C_B\|X_t\|_H^2 \right)\dif t+2\lambda_0Z_t\langle X_t, B(X_t)\dif W_t\rangle_H\\
\le &2\lambda_0Z_t\langle X_t, B(X_t)\dif W_t\rangle_H.
\end{align*}
Since $Z_t\ge0$, we obtain by Fatou's lamma $\EE^x[Z_t]\le \EE^x[Z_0]$, which is stronger than \eqref{eq exp1}.

The proof is complete.
\end{proof}

For any measurable set $K\subset H$, recall the  stopping times $\tau_K$ and $\tau_K^{(1)}$ defined by \eqref{stopping time}.
Now, we will verify the hyper-exponential recurrence condition \eqref{condition 2} in the following lemma.

\begin{lem} Under the conditions of Theorem \ref{thm main 01}, for any $\lambda>0$, there exists some compact set $K\subset  E$, such that
\begin{equation}\label{lem condition 2}
\sup_{x\in K}\EE^{x}\left[e^{\lambda\tau_K^{(1)}}\right]<\infty  \ \  \text{and} \ \ \ \sup_{\nu\in\mathcal M_{\lambda_0, L}}\EE^{\nu}\left[e^{\lambda\tau_K}\right]<\infty.
\end{equation}
\end{lem}
\begin{proof} The proof is inspirited by Gourcy \cite{Gou1, Gou2}.  Take
\Be\label{e: compact set}
K:=\left\{x\in V: \|x\|_V\le M\right\},
\Ee
where the constant $M$ will be fixed later. Since the embedding $V\subset H$ is compact, $K$ is a compact subset in $H$.

The definition of the occupation measure implies that
$$
\PP^{\nu}(\tau_K^{(1)}>n)\le \PP^{\nu}\left(\cL_n(K)\le  \frac1n\right)=\PP^{\nu}\left(\cL_n(K^c)\ge 1-\frac1n\right).
$$
With our choice for $K$, we have $ \|x\|_V\ge M1_{K^c}(x)$. Hence, for any fixed $\lambda_0\in (0,-\frac{c_3}{2B})$, we obtain by Chebychev's inequality
\begin{align*}
\PP^{\nu}(\tau_K^{(1)}>n)\le & \PP^{\nu}\left(\cL_n(\|x\|_V^r)\ge M^r \left(1-\frac1n\right)\right)\\
\le & \exp\left(-\frac{n \lambda_0 c_2 M^r}{2}\left(1-\frac1n\right) \right)\EE^{\nu}\left[\exp\left(\frac{\lambda_0 c_2}{2} \int_0^n \|X_s\|_V^r\dif s\right) \right].
\end{align*}
Integrating \eqref{eq exp1} w.r.t. $\nu(\dif x)$ and plugging it into the above estimate yields
$$
\PP^{\nu}(\tau_K^{(1)}>n)\le \nu (e^{\lambda_0\|\cdot\|_H^2})\exp\left\{-n\lambda_0C\right\}, \ \ \ \forall n\ge2,
$$
where $C:=\frac{ c_2}{4}M^r-c_1$.

Let $\lambda>0$ be fixed. By the  formula of  the integration by parts, we have
\begin{align*}
\EE^{\nu}\left[e^{\lambda \tau_K^{(1)}}\right]=&1+\int_0^{+\infty}\lambda e^{\lambda t}\PP^{\nu}(\tau_K^{(1)}>t)\dif t\\
\le & e^{2\lambda}+\sum_{n\ge2}\lambda e^{\lambda(n+1)}\PP^{\nu}(\tau_K^{(1)}>n)\\
\le & e^{2\lambda}\left(1+\lambda \nu (e^{\lambda_0\|\cdot\|_H^2})\sum_{n\ge2}e^{-n(\lambda_0C-\lambda)} \right).
\end{align*}
Now, we can choose $M$ such that $\lambda_0C-\lambda>0$ in the definition \eqref{e: compact set} of $K$. Then, taking the supremum over $\{\nu=\delta_x, x\in K\}$, we get
$$
\sup_{x\in K} \EE^x\left[e^{\lambda \tau_K^{(1)}}\right]\le e^{2\lambda}\left(1+\lambda e^{\frac{\lambda_0 M^2}{\eta^2}}  \sum_{n\ge2}e^{-n(\lambda_0C-\lambda)} \right)<\infty,
$$
where  \eqref{eq compact} is used. Thus,  the first inequality in \eqref{lem condition 2} holds true. We obtain the second inequality in \eqref{lem condition 2} in the same way: since $\tau_K\le \tau_K^{(1)}$, we have
\begin{align*}
\sup_{\nu\in \cM_{\lambda_0, L}}\EE^{\nu}\left[e^{\lambda \tau_K}\right]\le &
\sup_{\nu\in \cM_{\lambda_0, L}}\EE^{\nu}\left[e^{\lambda \tau_K^{(1)}}\right]\\
\le &e^{2\lambda}\left(1+ \lambda L \sum_{n\ge2}e^{-n(\lambda_0C-\lambda)} \right)\\
<&\infty.
\end{align*}
The proof is complete.
\end{proof}

\subsection{The proof of   Theorem \ref{thm main 2}}

\subsubsection{Some estimates} \label{section recur}

In this part, we will give some prior estimates, which are necessary for verifying the hyper-exponential recurrence condition \eqref{condition 2}.

Recall the following inequalities (see \cite{Xu13}):
%\Be \label{e:PoiInq}
%\|A^{\sigma_1} x\|_H \le  C_{\sigma_1, \sigma_2} \|A^{\sigma_2} x\|_H, \ \ \ \ \ \ \forall \ \sigma_1 \le \sigma_2 \ \ \forall x \in H;
%\Ee
\Be \label{e:eAEst}
\|A^{\sigma} e^{-At}\|_H \le C_\sigma t^{-\sigma}, \ \ \ \ \ \forall \ \sigma>0 \ \ \ \forall \ t>0;
\Ee
%\Be \label{e:NInnPro}
%\langle x, -N(x)\rangle_H \le \frac 14, \ \ \ \ \forall \ x \in H;
%\Ee
%\Be \label{e:NVEst}
%\|N(x)\|_V \leq C (\|x\|_V+\|x\|^3_V), \ \ \ \ \forall \ x \in V;
%\Ee
%\Be\label{e: NAH}
%\|AN(x)\|_H\le C(1+\|x\|_V^2)(1+\|Ax\|_H^2);
%\Ee
\Be \label{e:L4}
\|x\|^4_{L^4}\le \|x\|_V^2\cdot\|x\|_H^2\le \|x\|_V^4, \ \ \ \ \  \forall  \ x\in V;
\Ee
\Be \label{e:3H}
\|x^3\|_{H}\le C\|A^{\frac14}x\|_H^2\cdot\|x\|_H \le C\|x\|_V\|x\|_H^2 , \ \ \ \ \  \forall  \ x\in V.
\Ee

\vskip0.3cm
Let $Z_t$ be the following  Ornstein-Uhlenbeck process:
\Be\label{e:OUAlp}
\dif Z_t+A Z_t \dif t= \dif L_t, \ \ \ Z_0=0,
\Ee
where $L_t=\sum_{k \in \Z_*} \beta_k l_k(t) e_k$ is the $\alpha$-stable process defined in Eq. \eqref{e:XEqn}. It is well known that
$$
Z_t=\int_0^t e^{-A(t-s)} \dif L_s=\sum_{k \in \Z_{*}} z_{k}(t) e_k,
$$
where $$z_{k}(t)=\int_0^t e^{-\gamma_k(t-s)}
\beta_k \dif l_k(s).$$
\vskip0.3cm
The following maximal inequality can be found  in \cite[Lemma 3.1]{Xu13}.
\begin{lem} \label{l:ZEst}  For  any $T>0, 0 \leq \theta<\beta-\frac 1{2 \alpha}$ and any $0<p<\alpha$, we have
$$
\E \left[\sup_{0 \leq t \le T}\|A^\theta Z_t\|^p_{H}\right] \le CT^{p/\alpha},
$$
where $C$ depends on $\alpha,\theta, \beta, p$.
\end{lem}

\vskip0.3cm

Let $Y_t:=X_t-Z_t$. Then $Y_t$ satisfies the following equation:
\begin{equation}\label{e:Y}
\dif Y_t+A Y_t\dif t=N(Y_t+Z_t)\dif t, \ \ \ \ Y_0=x.
\end{equation}

\begin{lem}\label{l:Y}  For all $T>0$, we have
 \begin{equation}  \label{e:YtEstZt}
\sup_{t\in [T/2,T]}\|Y_t\|_H\le C(T)\left(1+\sup_{0\le t\le T}\|Z_t\|_{V} \right),
\end{equation}
 where the constant $C(T)$     does not depend on the initial value $Y_0=x$.
\end{lem}

\begin{proof}

By the chain rule, we obtain that
\begin{equation}\label{e:Y1}
\frac{\dif \|Y_t\|_H^2}{\dif t}+2\|Y_t\|_V^2=2\langle Y_t, N(Y_t+Z_t) \rangle.
\end{equation}
Using the following Young inequalities: for any $y,z\in L^4(\T; \R)$ and $C_1>0$, there exists $C_2>0$ satisfying that
\begin{equation*}
\begin{split}
& |\Ll  y, z^3 \Rr_H|=\left|\int_{\mathbb T} y(\xi) z^3 (\xi) \dif \xi\right| \le \frac{\int_{\mathbb T} y^4(\xi) \dif \xi}{C_1}+ C_2\int_{\mathbb T} z^{4}(\xi) \dif \xi, \\
& |\Ll  y^2, z^2 \Rr_H|=\left|\int_{\mathbb T} y^2(\xi) z^2 (\xi) \dif \xi\right| \le \frac{\int_{\mathbb T} y^4(\xi) \dif \xi}{C_1}+C_2\int_{\mathbb T} z^4(\xi) \dif \xi, \\
& |\Ll  y^3, z  \Rr_H|=\left|\int_{\mathbb T} y^3(\xi) z(\xi) \dif \xi\right| \le \frac{ \int_{\mathbb T} y^4(\xi) \dif \xi}{C_1}+ C_2\int_{\mathbb T} z^4(\xi) \dif \xi,
\end{split}
\end{equation*}
and using    H\"older inequality and the elementary inequality $2\sqrt a\le a/b+b$ for all $a,b>0$, we  obtain that there exists a constant $C\ge1$ satisfying that
$$
2\langle Y_t, N(Y_t+Z_t) \rangle\le -\|Y_t\|_{L^4}^4+C(1+\|Z_t\|_{L^4}^4).
$$
This  inequality, together with  Eq. \eqref{e:L4},  Eq. \eqref{e:Y1} and   H\"older inequality, implies that
\begin{equation}\label{e:Y2}
\frac{\dif \|Y_t\|_H^2}{\dif t}+2\|Y_t\|_V^2\le -\|Y_t\|_{H}^4+C\left(1+\|Z_t\|_V^4\right).
\end{equation}

For any $t\ge0$, denote $$h(t):= \|Y_t\|_H^2,\ \  \ K_T:=\sup_{0\le t\le T}\sqrt{C(1+\|Z_t\|_V^4)}\ge 1.$$
 By Eq. \eqref{e:Y2}, we have
$$
\frac{\dif h(t)}{\dif t} \le -h^2(t)+K_T^2, \ \ \ \forall t\in[0,T],
$$
with the initial value $h(0)=\|x\|_H^2\ge0$.

By the comparison theorem (e.g., the deterministic case of   \cite[Chapter VI, Theorem 1.1]{IW}),  we obtain that
\begin{equation}\label{e:h g}
h(t)\le g(t), \ \ \ \ \ \ \forall t\in[0,T],
\end{equation}
where the function $g$ solves the following equaiton
\begin{equation}\label{e:g0}
\frac{\dif g(t)}{\dif t}= -g^2(t)+K_T^2, \ \ \ \forall t\in[0,T],
\end{equation}
with the initial value $g(0)=h(0)$. The solution of  Eq. \eqref{e:g0} is
\begin{equation*}
g(t)=K_T+2K_T\left( \frac{g(0)+K_T}{g(0)-K_T}e^{2K_T t}-1\right)^{-1}, \ \  \ \forall t\in[0,T],
\end{equation*}
where  it is understood that  $g(t)\equiv K_T$ when  $g(0)=K_T$.  It is easy to show that
  %\begin{itemize}
   %\item[(1)]if  $g(0)\in [0,K_T]$, we have
  % $$
  % g(t)\le K_T, \ \ \ \forall t\in[0,T];
  % $$
     %  \item[(2)] if $g(0)>K_T$, noticing that $K_T\ge 1$, we have that for all $t\in[T/2,T]$,
       %   \begin{align*}
%g(t)\le& K_T+2K_T\left(  e^{2K_T t}-1\right)^{-1}\\
%\le &  K_T\left(1+2(   e^{T}-1)^{-1}\right).
     %  \end{align*}
 %\end{itemize}
% Therefore,
 for any initial value $g(0)$, we have
 $$
 g(t)\le K_T\left(1+2(   e^{T}-1)^{-1}\right), \ \ \ \forall t\in[T/2,T].
  $$
  This inequlity, together with Eq. \eqref{e:h g} and the definition of $K_T$, immediately implies the required estimate \eqref{e:YtEstZt}.

         The proof is complete.
\end{proof}

\vskip0.3cm

\begin{lem}\label{l:Y2}
For all $T>0$, $\delta \in (0,1/2)$ and  $p \in (0,\alpha/4)$, we have
$$
\E^{x}\left[\|Y_{T}\|^p_{H_\delta} \right]\le C_{T,\delta,p},
$$
where the constant $C_{T,\delta, p}$  does not depend on the initial  value $Y_0=x$.
\end{lem}
\begin{proof} Since
$$Y_{T}=e^{-AT/2} Y_{T/2}+\int_{T/2}^T e^{-A(T-s)} N(Y_s+Z_s) \dif s,$$
for any $\delta \in (0,1/2)$, by the inequalities \eqref{e:eAEst}-\eqref{e:3H} and Lemma \ref{l:Y}, there exists a constant $C=C_{T,\delta}$ (whose value may be different from line to line by convention) satisfied  that
\begin{align*}
\|Y_{T}\|_{H_\delta}
\le & C \|Y_{T/2}\|_H+C\int_{T/2}^T (T-s)^{-\delta} \|N(Y_s+Z_s)\|_H \dif s\notag \\
  \le& C \|Y_{T/2}\|_H+C\int_{T/2}^T (T-s)^{-\delta} (\|Y_s\|_H+\|Z_s\|_H+ \|Y_s^3\|_H+\|Z_s^3\|_H) \dif s \notag \\
 \le & C \|Y_{T/2}\|_H+C\int_{T/2}^T (T-s)^{-\delta} (\|Y_s\|_H+\|Z_s\|_V+\|Y_s\|_{V} \|Y_s\|^2_H+\|Z_s\|^3_V) \dif s \notag\\
  \le & C \left(1+\sup_{0 \le t \le T} \|Z_t\|^3_V\right)+C\int_{T/2}^T (T-s)^{-\delta} \|Y_s\|_{V} \|Y_s\|^2_H \dif s.
\end{align*}
Next, we estimate  the last term in above inequality: by Eq. \eqref{e:Y2} and Lemma \ref{l:Y} again, we have
\begin{align*}
&\int_{T/2}^T (T-s)^{-\delta} \|Y_s\|_{V} \|Y_s\|^2_H \dif s\\
 \le &C\left(1+\sup_{0 \le t \le T}\|Z_t\|^2_V\right)\int_{T/2}^T (T-s)^{-\delta} \|Y_s\|_{V} \dif s \\
\le& C\left(1+\sup_{0 \le t \le T}\|Z_t\|^2_V\right)\left(\int_{T/2}^T (T-s)^{-2\delta} \dif s\right)^{\frac 12} \left(\int_{T/2}^T\|Y_s\|^2_{V} \dif s\right)^{\frac 12}  \\
\le& C\left(1+\sup_{0 \le t \le T}\|Z_t\|^2_V\right) \left(\|Y_{T/2}\|^2_H+\int_{T/2}^T (1+\|Z_s\|^4_V)\dif s\right)^{\frac 12} \\
\le &  C\left(1 +\sup_{0 \le t \le T} \|Z_t\|^4_V\right).
\end{align*}
Hence, by Lemma \ref{l:ZEst}, we obtain that for any $p \in (0,\alpha/4)$,
$$
\E^{x}\left[\|Y_{T}\|^p_{H_\delta}\right] \le C_{T,\delta,p}.
$$
The proof is complete.
\end{proof}

By Lemma  \ref{l:ZEst} and Lemma  \ref{l:Y2},  we  obtain that
\begin{lem}\label{l: X}
For all $T>0$,   $\delta \in (0,1/2)$ and $p \in (0,\alpha/4)$, we have
$$
\E^{x}\left[\|X_{T}\|^p_{H_\delta}\right] \le C_{T,\delta,p}
$$
where the constant $C_{T,\delta,p}$   does not depend on the initial  value $X_0=x$.
\end{lem}

\subsubsection{The hyper-exponential Recurrence}

In this part, we will verify the hyper-exponential recurrence condition \eqref{condition 2}.
\vskip0.3cm

For any $\delta\in(0,1/2), M>0$, define the hitting time  of $\{X_n\}_{n\ge1}$:
\begin{equation}
\tau_M=\inf\{k\ge1:\|X_{k}\|_{H_{\delta}}\le M\}.
\end{equation}
 Let
 $$
 K:=\{x\in H_{\delta}: \|x\|_{H_{\delta}}\le M\}.
 $$
 Clearly, $K$ is compact in $H$. Recall the definitions of  $\tau_K$ and $\tau_K^{(1)}$ in \eqref{stopping time}. It is obvious  that
\begin{equation}
\tau_K\le \tau_M,\ \ \ \ \ \tau_K^{(1)}\le \tau_M.
\end{equation}
This fact, together with the following important theorem,  implies the hyper-exponential recurrence condition  \eqref{condition 2}.

\begin{theorem}\label{thm hyper-exp 2} For any $\lambda>0$,  there exists $M=M_{\lambda, \delta}$  such that
$$
\sup_{\nu\in\mathcal M_1(H)}\E^{\nu}[e^{\lambda\tau_M}]<\infty.
$$
\end{theorem}
\begin{proof}  Since we have the uniform estimate in  Lemma \ref{l: X}, the proof of this theorem is the same as that in Theorem \ref{thm hyper-exp}. We omit the detail here.  \end{proof}

%For any $n\in\N$, let
%$$
%B_n:=\left\{\|X_{j}\|_{H_{\delta}}>M; j=1,\cdots,n\right\}=\{\tau_M>n\}.
%$$
%By the Markov property of $\{X_n \}_{n\in\N}$, Chebychev's inequality and Lemma \ref{l: X}, we obtain that for any $\nu\in \M_1(H)$, $p \in (0,\alpha/4)$,
%\begin{align*}
%\PP^{\nu}(B_n)=&\PP^{\nu}(B_{n-1})\cdot\PP^{\nu}(B_n|B_{n-1})\\
%\le&\PP^{\nu}(B_{n-1})\cdot \frac{\E_{X_{n-1}}\left[\|X_n\|_{H_{\delta}}^p\right]}{M^p}\\
%\le & \PP^{\nu}(B_{n-1})\cdot\frac{C_{\delta,p}}{M^p},
%\end{align*}
%where $C_{\delta,p}$ is the constant in Lemma \ref{l: X} (taking $T=1$).

%By induction, we have for any $n\ge0$,
%$$
%\PP^{\nu}(\tau_M>n)=\PP^{\nu}(B_n)\le \left(\frac{C_{\delta,p}}{M^p}\right)^n.
%$$
%This inequality, together with Fubini's theorem, implies that for any $\lambda>0,\nu\in\M_1(H)$,
%\begin{align*}
%\E^{\nu}\left[ e^{\lambda\tau_M}\right]=&\int_0^{\infty}\lambda e^{\lambda t}\PP^{\nu}(\tau_M>t)\dif t\\
%\le &\sum_{n=0}^{\infty}\lambda e^{\lambda (n+1)}\PP^{\nu}(\tau_M>n)\\
%\le &\sum_{n=0}^{\infty}\lambda e^{\lambda (n+1)}\left(\frac{C_{\delta,p}}{M^p}\right)^n,
%\end{align*}
%which is finite as $M>(C_{\delta,p} e^{\lambda})^{1/p}$.

%The proof is complete.

\vskip0.5cm
\noindent{\bf Acknowledgments}: We would like to gratefully thank Armen Shirikyan for pointing out the key estimate \eqref{e:YtEstZt} for us. R. Wang thanks the  Faculty of Science and Technology, University of Macau, for finance support and hospitality.    He was supported by Natural Science Foundation of China 11301498, 11431014 and the Fundamental Research Funds for the Central Universities WK0010000048. J. Xiong was  supported by Macao Science and Technology Fund FDCT 076/2012/A3 and Multi-Year Research Grants of the University of Macau Nos. MYRG2014-00015-FST and MYRG2014-00034-FST. L. Xu is supported by the grants: MYRG2015-00021-FST and Science and Technology Development Fund, Macao S.A.R FDCT 049/2014/A1. All of the three authors are supported by the research project RDAO/RTO/0386-SF/2014.

\section{Appendix}

For any $p>1$, consider the following equation
\begin{equation}\label{e:g}
\frac{\dif g(t)}{\dif t}= -C_1 g^p(t)+C_2 , \ \ \ \forall t\ge0,
\end{equation}
where $C_i>0, i=1,2$,  the initial value $g(0)\ge0$.  Obviously,  $g$  is monotone and $\lim_{t\rightarrow \infty}g(t)=(C_2/C_1)^{1/p}$. Furthermore, we have the following estimate uniformly over the initial value.

\begin{lem}\label{lem p}
For any $p>1$, there exists a constant $C(p,  C_1)$ such that for   any initial value $g(0)\ge0$, we have
\begin{equation}\label{eq p}
\sup_{t\ge 1} |g(t)|\le   C(p,  C_1)(1+ C_2).
\end{equation}

\end{lem}
\begin{proof} We  shall divide the proof in the following two cases.
  \begin{itemize}
    \item[(1)]When $g(0)\in \left[0,(C_2/C_1)^{1/p}\right]$,   $g$ is increasing and its limit is $(C_2/C_1)^{1/p}$, which implies that $g(t)\le (C_2/C_1)^{1/p}$ for all $t\ge0$.
    \item[(2)]
      When $g(0)>(C_2/C_1)^{1/p}$,   $g$ is decreasing  in $[0,\infty)$.  Let
      $$\tau:=\inf\left\{t: g(t)\le (2C_2/C_1)^{1/p} \right\}.$$
\begin{itemize}
  \item[(2a)]
      If $\tau\le 1$, there exists a constant $t_0\le 1$ such that $g(t_0)\le (2C_2/C_1)^{1/p} $,      by the monotone of $g$, we know that
       $$  g(t)\le (2C_2/C_1)^{1/p}\ \ \ \ \text{for all } t\ge t_0.$$

  \item[(2b)] If $\tau> 1$,  $g(t)>(2C_2/C_1)^{1/p} $ for any $t\le 1$, and
    $$
    g'(t)\le -C_1 g^p(t)/2\ \ \  \ \text{for all } t\le 1.
         $$
  By the comparison's theorem and the monotone of $g$, we obtain that for any $t\ge1$,
  $$
  g(t)\le g(1)\le \left[g^{-p+1}(0)+(p-1)C_1/2\right]^{-\frac{1}{p-1}}\le \left[(p-1)C_1 /2\right]^{-\frac{1}{p-1}}.
  $$
  \end{itemize}
   \end{itemize}
   Above all, there exists a constant $C(p,  C_1)$ independent of $g(0)$ satisfying \eqref{eq p}.

  The proof is complete.
  \end{proof}

%%%%%%%%%%%%%%%%%%%%%%%%%%%%%%%%%%%%%%%%%%

\bibliographystyle{amsplain}

\end{document}